\documentclass[11pt,reqno]{amsart}
\usepackage[utf8]{inputenc}
\usepackage{amsfonts,amsthm,amsmath,amssymb,thmtools}
\usepackage{booktabs}
\usepackage{boldline,makecell}
\usepackage{array,diagbox}
\usepackage{graphicx}
\usepackage[dvipsnames,table]{xcolor}   
\usepackage{enumerate}
\usepackage{hyperref}
\usepackage{color}
\usepackage{tikz-cd}
\usepackage[margin=1in]{geometry}
\usepackage[
    maxbibnames=99,
    backend=biber,
    style=alphabetic,
    giveninits=true
]{biblatex}
\DeclareFieldFormat{pages}{#1}
\renewbibmacro{in:}{%
  \ifentrytype{article}
    {}
    {\bibstring{in}%
     \printunit{\intitlepunct}}}
\DeclareFieldFormat
[article,inbook,incollection,inproceedings,patent,thesis,unpublished]
  {title}{\mkbibemph{#1}}
\DeclareFieldFormat{journaltitle}{#1\isdot}
\DeclareFieldFormat[article]{volume}{\mkbibbold{#1}}
\DeclareFieldFormat[article]{number}{\bibstring{number}\addnbspace #1}

\renewbibmacro*{journal+issuetitle}{%
  \usebibmacro{journal}%
  \setunit*{\addspace}%
  \iffieldundef{series}
    {}
    {\newunit
     \printfield{series}%
     \setunit{\addspace}}%
  \printfield{volume}%
  \setunit{\addspace}%
  \usebibmacro{issue+date}%
  \setunit{\addcomma\space}%
  \printfield{number}%
  \setunit{\addcolon\space}%
  \usebibmacro{issue}%
  \setunit{\addcomma\space}%
  \printfield{eid}
  \newunit}

\newtheoremstyle{mystyle}
  {}
  {}
  {\itshape}
  {}
  {\bfseries}
  {.}
  { }
  {\thmname{#1}\thmnumber{ #2}\thmnote{ (#3)}}

\theoremstyle{mystyle}
\newtheorem*{Thm*}{Theorem}
\newtheorem{Thm}{Theorem}
\newtheorem{Lem}[Thm]{Lemma}

\newtheorem{Prop}[Thm]{Proposition}
\newtheorem*{Conj*}{Conjecture}
\newtheorem{Que}[Thm]{Question}

\theoremstyle{definition}

\newtheorem{Ex}[Thm]{Example}

\theoremstyle{remark}
\newtheorem{Rmk}[Thm]{Remark}

\declaretheoremstyle[%
  spaceabove=3pt,
  spacebelow=10pt,
  headfont=\normalfont\itshape,%
  postheadspace=.5em,%
  qed=\qedsymbol%
]{mystyle2} 
\declaretheorem[name={Proof},style=mystyle2,unnumbered,
]{pf}

\newcommand{\Z}{\mathbb{Z}}
\newcommand{\Q}{\mathbb{Q}}
\newcommand{\CP}{\mathbb{CP}}
\newcommand{\RP}{\mathbb{RP}}
\DeclareMathOperator{\rank}{\mathrm{rank}}

\title{Slice genus bound in $DTS^2$ from $s$-invariant}

\author{Qiuyu Ren}
\address{Department of Mathematics, University of California, Berkeley, Berkeley, CA 94709, USA}
\email{qiuyu\_ren@berkeley.edu}
\begin{document}

\begin{abstract}
We prove a recent conjecture of Manolescu-Willis which states that the $s$-invariant of a knot in $\RP^3$ (as defined by them) gives a lower bound on its null-homologous slice genus in the unit disk bundle of $TS^2$. We also conjecture a lower bound in the more general case where the slice surface is not necessarily null-homologous, and give its proof in some special cases.
\end{abstract}

\maketitle

\section{Introduction}
Rasmussen \cite{rasmussen2010khovanov} famously defined the $s$-invariant for knots in $S^3$ using Khovanov homology theory \cite{khovanov2000categorification}, and proved that for a knot $K$ in $S^3$, 
\begin{equation}\label{eq:g_4}
2g_4(K)\ge|s(K)|,
\end{equation}
where $g_4(K)$ is the slice genus of $K$, which can be defined as the minimal genus of an orientable cobordism (in $S^3\times[0,1]$) from $K$ to the unknot.

Analogously, Manolescu-Marengon-Sarkar-Willis \cite{manolescu2023generalization} and Manolescu-Willis \cite{manolescu2023rasmussen} defined $s$-invariants ($\Z$-valued, like the usual $s$-invariant) for null-homologous knots in $S^1\times S^2$ and for all knots in $\RP^3$, respectively, and proved the same inequality \eqref{eq:g_4} in these settings. Here the slice genus $g_4(K)$ for $K$ is still defined as the minimal genus of an orientable cobordism from $K$ to an unknot (there are two unknots in $\RP^3$, one null-homologous and one not).

For an integer $d$, let $D(d)$ denote the $D^2$-bundle over $S^2$ with euler number $d$. Thus $D(0)=D^2\times S^2$ with boundary $S^1\times S^2$; $D(1)=\CP^2\backslash B^4$ with boundary $S^3$; $D(2)=DTS^2$, the unit disk bundle of the tangent bundle of $S^2$, with boundary $\RP^3$. For a null-homologous properly embedded orientable connected surface $\Sigma\subset D(d)$ with boundary a knot $K\subset\partial D(d)$, $d=0,1,2$, the genus bound
\begin{equation}\label{eq:g_D}
2g(\Sigma)\ge-s(K)
\end{equation}
was proved for $d=0,1$ \cite[Theorem~1.15,Corollary~1.9]{manolescu2023generalization} and conjectured for $d=2$ \cite[Conjecture~6.9]{manolescu2023rasmussen}. We prove this conjecture.

\begin{Thm*}
If $(\Sigma,K)\subset(DTS^2,\RP^3)$ is a null-homologous properly embedded orientable connected surface that bounds a knot $K\subset\RP^3$, then $$2g(\Sigma)\ge-s(K).$$
\end{Thm*}

\begin{Rmk}
By reversing the orientation of $D(d)$, \eqref{eq:g_D} for $d=0,1,2$ implies $$2g(\Sigma)\ge s(K)$$ for $d=0,-1,-2$. Since a cobordism in $\partial D(d)\times[0,1]$ from a null-homologous $K\subset\partial D(d)$ to an unknot can be capped off to become a slice surface in $D(\pm d)$, \eqref{eq:g_D} for $d=0,1,2$ can be considered as refinements of \eqref{eq:g_4} for null-homologous knots $K$ in $S^1\times S^2,S^3,\RP^3$, respectively. In the only other case, namely when $K\subset\RP^3$ is not null-homologous, \eqref{eq:g_4} is refined by \eqref{eq:g_D2} below with $[\Sigma]=\pm1$.
\end{Rmk}

When $d=0$, $\Sigma\subset D^2\times S^2$ being null-homologous is equivalent to $\partial\Sigma\subset S^1\times S^2$ being null-homologous. Since the $s$-invariant is only defined for null-homologous knots in $S^1\times S^2$, the null-homologous condition on $\Sigma$ puts no restriction. When $d=1,2$, however, we could ask whether $s(K)$ gives genus bounds for slice surfaces of $K$ in $D(d)$ that are not necessarily null-homologous. For $d=1$ this is conjectured in \cite{manolescu2023generalization} and proved by Ren \cite[Corollary~1.5]{ren2023lee}. Explicitly, for any $(\Sigma,K)\subset(\CP^2\backslash B^4,S^3)$ we have 
\begin{equation}\label{eq:g_D1}
2g(\Sigma)\ge-s(K)-[\Sigma]^2+|[\Sigma]|,
\end{equation}
where $|\cdot|$ is the $L^1$-norm (equivalently, absolute value) on $H_2(\CP^2\backslash B^4,S^3)\cong\Z$. We pose the following conjecture for the case $d=2$.

\begin{Conj*}
If $(\Sigma,K)\subset(DTS^2,\RP^3)$ is a properly embedded orientable connected surface that bounds a knot $K\subset\RP^3$, then 
\begin{equation}\label{eq:g_D2}
2g(\Sigma)\ge-s(K)-\frac{[\Sigma]^2}{2}.
\end{equation}
\end{Conj*}
The main theorem shows the conjecture when $[\Sigma]=0$. In fact, the same proof applies to show it in a couple more cases.
\begin{Prop}\label{prop:1}
The inequality \eqref{eq:g_D2} holds if $[\Sigma]=\pm1,\pm2,\pm3$ in $H_2(DTS^2,\RP^3)\cong\Z$.
\end{Prop}

In fact, in the various settings above, the $s$-invariants are defined for links as well as knots \cite{beliakova2008categorification,manolescu2023generalization,manolescu2023rasmussen}. As remarked in \cite{manolescu2023generalization,manolescu2023rasmussen}, \eqref{eq:g_D} for $d=0,1,2$ and \eqref{eq:g_D1} reduce to computing the $s$-invariants of some special family of links; similarly, \eqref{eq:g_D2} also reduces to computing the $s$-invariant of a family of links. We will explain these reductions in Section~\ref{sec:reduce}. In Section~\ref{sec:s}, we calculate the $s$-invariants that enable us to conclude the main theorem and Proposition~\ref{prop:1}. We also pose a technical question whose positive answer implies the conjecture above.

\subsection*{Acknowledgement}
The author thanks Ciprian Manolescu for suggesting this problem and for helpful discussions.

\section{Reduce to \texorpdfstring{$s$}{s}-invariant calculations}\label{sec:reduce}
In this section, we define links $T(d;p,q)\subset\partial D(d)$ for $p,q\ge0$, such that \eqref{eq:g_D} reduces to the calculation of the $s$-invariants of $T(d;p,p)$, $d=0,1,2$, and \eqref{eq:g_D1}\eqref{eq:g_D2} reduce to that of $T(d;p,q)$, $d=1,2$. The strategy is essentially due to \cite{manolescu2023generalization,manolescu2023rasmussen}, but we carry it out explicitly for completeness.

Let $S\subset D(d)$ denote the core $2$-sphere of the $D^2$-bundle $D(d)\to S^2$. Any properly embedded oriented surface $\Sigma\subset D(d)$ can be perturbed so that it intersects $S$ transversely at some $p$ points positively and some $q$ points negatively. Removing a tubular neighborhood of $S$, we obtain a properly embedded surface $\Sigma_0\subset\partial D(d)\times[0,1]$, whose boundary on $\partial D(d)\times\{1\}$ is the original boundary of $\Sigma$ and whose boundary on $\partial D(d)\times\{0\}$ is an oriented link in $\partial D(d)$ consisting of $p+q$ fibers of the $S^1$-bundle $\partial D(d)\to S^2$, $p$ of which oriented positively and $q$ of which negatively. We denote the mirror of this link by $T(d;p,q)$; thus, $\Sigma_0$ is a cobordism from $\overline{T(d;p,q)}$ to $\partial\Sigma$.

\begin{Ex}
$T(0;p,q)$ is the disjoint union of $p+q$ knots of the form $*\times S^2\subset S^1\times S^2$, $p$ of which oriented upwards and $q$ of which downwards. It is denoted by $F_{p,q}$ in \cite{manolescu2023generalization}.
\end{Ex}

\begin{Ex}
$\overline{\partial D(1)}\to S^2$ is the Hopf fibration, hence its fibers have pairwise linking number $1$. Thus $T(1;p,q)$ is the torus link $T(p+q,p+q)$ in which $p$ of the strands are oriented against the other $q$ strands. It is denoted by $F_{p,q}(1)$ in \cite{manolescu2023generalization} and $T(p+q,p+q)_{p,q}$ in \cite{ren2023lee}.
\end{Ex}

\begin{Ex}\label{ex:2}
Think of $\RP^3$ as the $3$-ball $B^3$ with antipodal points on the boundary identified. Then $T(2;p,q)$ can be obtained by standardly embedding a half-twist on $p+q$ strands, $p$ of which oriented against the other $q$, into $B^3$ such that the endpoints land on the boundary. This can be seen by realizing $T(2;p,q)\subset\RP^3$ as the quotient of $T(1;p,q)\subset S^3$ by the standard involution on $S^3$. $T(2;p,p)$ is denoted by $H_p$ in \cite{manolescu2023rasmussen}.
\end{Ex}

By \cite{beliakova2008categorification,manolescu2023generalization,manolescu2023rasmussen}, if $\Sigma$ is an oriented cobordism in $Y\times[0,1]$ between two (null-homologous if $Y=S^1\times S^2$) oriented links $L_0$ and $L_1$ in $Y$, $Y=S^3,S^1\times S^2,\RP^3$, such that every component of $\Sigma$ has a boundary in $L_0$, then 
\begin{equation}\label{eq:s_cob}
s(L_1)-s(L_0)\ge\chi(\Sigma).
\end{equation}
By construction, if $(\Sigma,L)\subset(D(d),\partial D(d))$ is a properly embedded oriented connected surface without closed components, by deleting a tubular neighborhood of the core $S\subset D(d)$, we obtain a cobordism $\Sigma_0$ from some $\overline{T(d;p,q)}$ to $L$, each of whose components has a boundary in $L$. Turning the cobordism upside down and applying \eqref{eq:s_cob} give
\begin{equation}\label{eq:adj_pre}
s(T(d;p,q))-s(\bar L)\ge\chi(\Sigma_0)=\chi(\Sigma)-p-q,
\end{equation}
where the last inequality holds because topologically $\Sigma_0$ is $\Sigma$ with $p+q$ disks removed from its interior.

The number $p-q$ equals to the homology class $[\Sigma]\in H_2(D(d),\partial D(d))$ upon an identification $H_2(D(d),\partial D(d))\cong H^2(D(d))\cong\Z$. Thus, for $[\Sigma]=p-q$ a fixed number, if $s(T(d;p,q))+p+q$ is independent of specific $p,q$, then \eqref{eq:adj_pre} can be rewritten in terms of $s(\bar L)$, $\chi(\Sigma)$, and $[\Sigma]$. This is the case for $d=0$, $[\Sigma]=0$, as well as for $d=1$ with any $[\Sigma]$. Explicitly, by \cite[Theorem~1.6]{manolescu2023generalization}\cite[Theorem~1.1]{ren2023lee} we have 
\begin{equation}\label{eq:s_01}
s(T(0;p,p))=-2p+1,\ s(T(1;p,q))=(p-q)^2-2\max(p,q)+1.
\end{equation}
We conjecture this is also true for $d=2$ and any $[\Sigma]$, with
\begin{equation}\label{eq:s_2}
s(T(2;p,q))=\left\lfloor\frac{(p-q)^2}{2}\right\rfloor-p-q+1,
\end{equation}
and give a proof of it for $[\Sigma]=0,\pm1,\pm2,\pm3$. 

In the special case when $L=K$ is a knot, we have $s(\bar L)=s(\bar K)=-s(K)$ by \cite[Proposition~3.10]{rasmussen2010khovanov}\cite[Proposition~8.8(1)]{manolescu2023generalization}\cite[Proposition~4.10]{manolescu2023rasmussen}. In this case,
plugging \eqref{eq:s_01} into \eqref{eq:adj_pre} gives \eqref{eq:g_D} for $d=0,1$, and \eqref{eq:g_D1}. Plugging \eqref{eq:s_2} into \eqref{eq:adj_pre} would give the conjectural inequality \eqref{eq:g_D2}, although we are only able to prove it for $|p-q|\le3$.

\begin{Rmk}
\begin{enumerate}
\item It is easy to prove \eqref{eq:s_2} for $pq=0$, since in this case $T(2;p,q)$ is a positive link and one can apply \cite[Remark~6.3]{manolescu2023rasmussen}. However this does not help in establishing \eqref{eq:g_D2}.
\item If \eqref{eq:s_2} were true in general, one can proceed as in \cite[Section~4]{ren2023lee} to determine the entire quantum filtration structure of the Lee homology (as defined in \cite{manolescu2023rasmussen}) of $s(T(2;p,q))$.
\end{enumerate}
\end{Rmk}

\section{\texorpdfstring{$s$}{s}-invariants of \texorpdfstring{$T(2;p,q)$}{T(2;p,q)}}\label{sec:s}
As explained in Section~\ref{sec:reduce}, the main theorem and Proposition~\ref{prop:1} reduce to the following proposition.
\begin{Prop}\label{prop:s}
For $p,q\ge0$ with $|p-q|\le3$, the $s$-invariant (as defined in \cite{manolescu2023rasmussen}) of the link $T(2;p,q)\subset\RP^3$ defined in Section~\ref{sec:reduce} is given by $$s(T(2;p,q))=\left\lfloor\frac{(p-q)^2}{2}\right\rfloor-p-q+1.$$
\end{Prop}
Sto\v{s}i\'c \cite[Theorem~3]{stovsic2009khovanov} calculated the Khovanov homology groups of the positive torus links $T(n,n)$ in their highest nontrivial homological grading $h=2n^2$. For dimension reason, the Lee spectral sequence from $Kh(T(n,n))\otimes\Q$ to the Lee homology $Kh_{Lee}(T(n,n))$ collapses immediately in this homological degree. This can be used to give an alternative proof of $s(T(1;p,p))=1-2p$, a fact that is reproved in \cite[Theorem~1.7]{manolescu2023generalization}. We prove Proposition~\ref{prop:s} by adapting the argument of Sto\v{s}i\'c. It is worth remarking that the calculation of the more general $s(T(1;p,q))$ was done in \cite{ren2023lee} by pushing Sto\v{s}i\'c's argument slightly further. However, we were not able to achieve the same here to prove \eqref{eq:s_2} in its full generality.

\subsection{Review of Khovanov homology in \texorpdfstring{$\RP^3$}{RP3}}
We first briefly review some properties of the Khovanov/Lee homology and the $s$-invariant of links in $\RP^3$, following \cite{manolescu2023rasmussen}. We only give definitions that will be relevant to us. We assume the reader is familiar with the usual theory in $S^3$, in particular \cite{khovanov2000categorification,lee2005endomorphism,rasmussen2010khovanov}.

Think $\RP^3\backslash*$ as the twisted $I$-bundle over $\RP^2$, we see that links in $\RP^3$ can be represented by link diagrams in $\RP^2$, and two different diagrams of the same link are related by the usual three Reidemeister moves. Although the over/under strands are not well-defined at a crossing, it is unambiguous to distinguish positive/negative crossings (if the link is oriented or has only one component) or to define $0/1$-resolutions at a crossing in such a link diagram.

Let $L\subset\RP^3$ be an oriented link with an oriented link diagram $D$. Let $2^n=(0\to1)^n$ denote the hypercube of complete resolutions seen as a directed graph, where $n$ is the number of crossings in $D$. Every vertex $v$ corresponds to a complete resolution $D_v$, which is assigned a bigraded abelian group $C(D_v)$. Every edge $e$ from a vertex $v$ to a vertex $w$ corresponds to a saddle from $D_v$ to $D_w$, which is assigned four maps $\partial_0^e,\partial_-^e,\Phi_0^e,\Phi_+^e\colon C(D_v)\to C(D_w)$ of bidegree $(0,0),(0,-2),(4,0),(4,2)$, respectively. The \textit{Khovanov complex} of $D$ is $$C(D):=\oplus_vC(D_v)[-n_-+|v|]\{n_+-2n_-\}$$ equipped with the differential $\partial:=\sum_e\partial_0^e$. Here $[\cdot]$ denotes the homological grading shift, $\{\cdot\}$ denotes the shift in the first grading (called quantum grading) of $C(D_v)$, $n_\pm$ denotes the number of positive/negative crossings in $D$ and $|v|$ denotes the number of $1$'s in $v$. The \textit{Lee complex} is $C_{Lee}(D):=C(D)\otimes\Q$ equipped with the differential $\partial_{Lee}:=\sum_e(\partial_0^e+\partial_-^e+\Phi_+^e+\Phi_0^e)\otimes\Q$. For our purpose, we also consider a \textit{deformed Khovanov complex}, defined as $C'(D):=C(D)$ equipped with the differential $\partial':=\sum_e(\partial_0^e+\partial_-^e)$. The cohomologies of these three complexes are denoted by $Kh(L)$, $Kh_{Lee}(L)$, $Kh'(L)$, respectively, which do not depend on the choice of the link diagram $D$.

The group $Kh(L)$ is trigraded by $h$ (homological grading), $q$ (quantum grading), and $k$; $Kh'(L)$ is bigraded by $h,q$; $Kh_{Lee}(L)$ is graded by $h$ and filtered by $q$. As a vector space, $Kh_{Lee}(L)\cong\Q^{2^{|L|}}$ is spanned by some generators $[s_\mathfrak{o}]$, where $\mathfrak{o}$ runs over the all possible orientations of $L$ as an unoriented link\footnote{In fact, the definition of $[s_\mathfrak{o}]$ depends on some auxiliary choices, which we may ignore here.}. When $\mathfrak{o}$ is the given orientation on $L$, $[s_\mathfrak{o}]$ sits in homological degree $0$, and its quantum filtration degree plus $1$ is defined as the \textit{$s$-invariant} of $L$. As a filtered complex, the associated graded complex of $C_{Lee}(L)$ is exactly $C'(L)\otimes\Q$. Thus, there is a spectral sequence with $E_1$-page $Kh'(L)\otimes\Q$ that converges to $Kh_{Lee}(L)$, whose $r$-th differential has bidegree $(1,4r)$.

The orientation on $L$ plays only a minor role on the group $Kh^\bullet(L)$, where $\bullet$ denotes one of the three favors we are considering. Explicitly, negating the orientation on a sublink $L'\subset L$ shifts its grading by $[2\ell]\{6\ell\}$, where $\ell$ is the linking number (which takes half-integer values) between $L'$ and $L\backslash L'$ with the new orientations.

Every cobordism $\Sigma\colon L_0\to L_1$ between two oriented links with diagrams $D_0,D_1$ induces a chain map $C^\bullet(\Sigma)\colon C^\bullet(D_0)\to C^\bullet(D_1)$ with some grading shifts. By design, if $D_{0/1}$ are the $0/1$-resolutions at a crossing of a link diagram $D$ of some link $L$, and $\Sigma$ is the obvious saddle cobordism, then $C^\bullet(D)$ is isomorphic to the mapping cone of $C^\bullet(\Sigma)$ up to grading shifts. More explicitly, for our convenience we record that if $D_{0/1}$ are the $0/1$-resolutions of $D$ at a positive crossing, and $L_0$ is assigned the induced orientation from $L$ while $L_1$ is assigned any orientation, then $C'(D)\cong Cone(C'(D_0)\to C'(D_1)[c]\{3c+1\})[1]\{1\}$, where $c=n_-(D_1)-n_-(D)$. Thus we have the following exact triangle of deformed Khovanov homology groups
\begin{equation}\label{eq:skein}
\begin{tikzcd}
Kh'(L_1)[c+1]\{3c+2\}\ar[rr]&&Kh'(L)\ar[ld]\\&Kh'(L_0)\{1\}\ar[lu,"\text{[}1\text{]}"].&
\end{tikzcd}
\end{equation}

If $\Sigma\colon L_0\to L_1$ is an oriented cobordism, it preserves the homological grading $h$ and changes the quantum grading $q$ by $\chi(\Sigma)$. Moreover, it sends a generator $[s_\mathfrak{o}]\in Kh_{Lee}(L_0)$ to some $\sum_{\mathfrak{o'}}\lambda_{\mathfrak{o'}}[s_\mathfrak{o'}]\in Kh_{Lee}(L_1)$, where $\mathfrak{o'}$ runs over orientations of $L_1$ such that there is an orientation on $\Sigma$ making it an oriented cobordism $(L_0,\mathfrak{o})\to(L_1,\mathfrak{o'})$, and $\lambda_{\mathfrak{o'}}\in\Q^\times$. In particular, this implies \eqref{eq:s_cob}.

Finally, we remark that there are by definition two unknots in $\RP^3$. The \textit{class-$0$ unknot} $U_0$ is an unknot in a small ball contained in $\RP^3$; the \textit{class-$1$ unknot} $U_1$ is a copy of the standardly embedded $\RP^1\subset\RP^3$. Both these unknots have rank-$2$ deformed Khovanov homology given by $Kh'^{0,\pm1}(U_i)=\Z$. The deformed Khovanov homology behaves as expected under the disjoint union of two links, one in $\RP^3$ and one in $S^3$. In particular, regarding $U_0$ as a knot in $S^3$, we have $Kh'(L\sqcup U_0)=Kh'(L)\{1\}\oplus Kh'(L)\{-1\}$ for any link $L\subset\RP^3$.

\subsection{Calculation of \texorpdfstring{$s$}{s}}
Now we are ready to prove Proposition~\ref{prop:s}. We first define two auxiliary family of links $T_n^i$, $S_n^i$, $0\le i\le n-1$. These should be compared with $D_{n,n-1}^i$, $D_{n,n}^i$ in \cite[Section~5]{ren2023lee}.

Think of $\RP^2$ as $D^2$ with antipodal points on the boundary identified. A braid diagram can be placed into $D^2$ with its endpoints on $\partial D^2$, identified pairwisely to give a link diagram in $\RP^2$; this is called by \cite{manolescu2023rasmussen} the projective closure of the given braid. Let $T_n^i$ be the link represented by the projective closure of $\sigma_{n-1}(\sigma_{n-2}\sigma_{n-1})\cdots(\sigma_2\cdots\sigma_{n-1})(\sigma_1\cdots\sigma_i)$, and $S_n^i$ be the link presented by the projective closure of $\sigma_{n-1}(\sigma_{n-2}\sigma_{n-1})\cdots(\sigma_1\cdots\sigma_{n-1})(\sigma_{n-1}\cdots\sigma_{n-i})$. We equip $T_n^i$, $S_n^i$ with the orientation where all strands are oriented upwards. By the description in Example~\ref{ex:2}, $T(2;n,0)$ is exactly the link $T_n^{n-1}$, and all $T(2;p,n-p)$ are $T_n^{n-1}$ with a possibly different orientation. Also, by definition $S_n^0=T_n^{n-1}$, and it is easy to check by Reidemeister moves\footnote{If one wishes to think the link diagrams as sitting in $D^2$, there will be two additional Reidemeister moves when one crosses the boundary, as illustrated in \cite[Figure~1]{gabrovvsek2013categorification}. Note however the picture (e) there was incorrectly drawn.} that $T_n^0=S_{n-2}^{n-3}$.

For $n\ge2$, $i>0$, resolving the crossing in the standard diagram of $T_n^i$ that corresponds to the last word $\sigma_i$ gives $T_n^{i-1}$ as the $0$-resolution and some other link $(S_n^i)_1$ as the $1$-resolution. Similarly, resolving the crossing of $S_n^i$ that corresponds to the last $\sigma_{n-i}$ gives $0$-resolution $S_n^{i-1}$ and $1$-resolution $(S_n^i)_1$. By Reidemeister moves, one may check that in fact (as unoriented links)
\begin{equation}\label{eq:T_1S_1}
(T_n^i)_1=\begin{cases}T_{n-2}^{n-3}\sqcup U_0,&i=n-1\\T_{n-2}^{i-1},&i<n-1\end{cases},\ (S_n^i)_1=\begin{cases}S_{n-2}^{i-2},&i>1\\S_{n-2}^0\sqcup U_0,&i=1\end{cases}.
\end{equation}
Here $U_0$ is the class-$0$ unknot, and as a convention we define $T_0^{-1}=S_0^0=\emptyset$ to be the empty link.

Give $(T_n^i)_1$, $(S_n^i)_1$ the orientations of the right hand sides in the identification \eqref{eq:T_1S_1}, the skein exact triangle \eqref{eq:skein} gives exact triangles
\begin{equation}\label{eq:LES_T}
Kh'((T_n^i)_1)[n-1]\{3n-4\}\to Kh'(T_n^i)\to Kh'(T_n^{i-1})\{1\}\xrightarrow{[1]},
\end{equation}
\begin{equation}\label{eq:LES_S}
Kh'((S_n^i)_1)[n]\{3n-1\}\to Kh'(S_n^i)\to Kh'(S_n^{i-1})\{1\}\xrightarrow{[1]}.
\end{equation}

We prove a lemma that gives a ``graphical lower bound'' of the deformed Khovanov homology groups of $T_n^i$ and $S_n^i$, in the spirit of \cite[Theorem~2.1]{ren2023lee}. In fact, most parts of the statement won't be relevant for our purpose. But since the general statement is not much more complicated to state and to prove, we include it fully here.
\begin{Lem}\label{lem:bound}
\begin{enumerate}
\item $Kh'^{h,q}(T_n^i)=0$ for $h<0$ or $h>\lfloor n^2/4\rfloor$ or $q-h<\lfloor n^2/2\rfloor-2n+1+i$.
\item $Kh'^{h,q}(S_n^i)=0$ for $h<0$ or $h>\lfloor n^2/4\rfloor+\lceil i/2\rceil$ or $q-h<\lfloor n^2/2\rfloor-n+i$ or $q-2h<\lfloor n^2/4\rfloor-n+i$.
\end{enumerate}
\end{Lem}
\begin{Rmk}\label{rmk:fail}
The $s$-invariant of all $T(1;p,q)\subset S^3$ was deduced in \cite{ren2023lee} from Theorem~2.1 there. The difficulty that prevents us to similarly deduce \eqref{eq:s_2} from Lemma~\ref{lem:bound} is that we were not able to establish an injectivity result like the addendum in Theorem~2.1(1) in \cite{ren2023lee} (there we actually have an isomorphism; however, only injectivity is needed for the proof, and only injectivity is expected in our case). See also Section~\ref{sec:s_que}.
\end{Rmk}
\begin{pf}
We induct on $n,i$. For $n=1$ this is immediate, since $S_1^0=T_1^0=U_1$ is the class-$1$ unknot. For $n=2$, $T_2^0=U_0$ satisfies \textit{(1)}. By \eqref{eq:T_1S_1}\eqref{eq:LES_T}, nonzero homology groups of $T_2^1$ are exactly $$Kh'^{h,q}(T_2^1)=\Z,\ (h,q)=(0,0),(0,2),(1,1),(1,3),$$ thus $T_2^1$ satisfy \textit{(1)} and $S_2^0=T_2^1$ satisfy \textit{(2)}. By \eqref{eq:T_1S_1}\eqref{eq:LES_S}, $S_2^1$ has $$Kh'^{h,q}(S_2^1)=\Z,\ (h,q)=(0,1),(0,3),(1,2),(2,6),$$ and zero elsewhere possibly except when $(h,q)=(1,4),(2,4)$, thus it satisfies \textit{(2)}. Now, by the induction hypothesis and \eqref{eq:T_1S_1}\eqref{eq:LES_T}\eqref{eq:LES_S}, \textit{(1)(2)} are inductively proved for $n>2$ by checking the elementary statements that\vspace{-5pt}
\begin{itemize}
\item The vanishing region (in the $hq$-coordinate plane) of $Kh'(T_n^0)$ described in \textit{(1)} is contained in that of $Kh'(S_{n-2}^{n-3})$ described in \textit{(2)};
\item For $i>0$, the vanishing region of $Kh'(T_n^i)$ is contained in that of both $Kh'((T_n^i)_1)[n-1]\{3n-4\}$ and $Kh'(T_n^{i-1})\{1\}$;
\item The vanishing region of $Kh'(S_n^0)$ is identical to that of $Kh'(T_n^{n-1})$;
\item For $i>0$, the vanishing region of $Kh'(S_n^i)$ is contained in that of both $Kh'((S_n^i)_1)[n]\{3n-1\}$ and $Kh'(S_n^{i-1})\{1\}$.\qedhere
\end{itemize}
\end{pf}

\begin{proof}[Proof of Proposition~\ref{prop:s}]
We divide into four cases according to the value of $|p-q|$. We give the proof carefully for the case $|p-q|=0$, and more casually for the rest cases, as they will be similar.

\textbf{Case 1}: $|p-q|=0$.

Write $n=p+q=2m$. We induct on $m$ to show that 
\begin{equation}\label{eq:rank}
\rank Kh'^{m^2,*}(T_n^i)=2\binom{i}{m},
\end{equation}
\begin{equation}\label{eq:low}
\inf\{q\colon Kh'^{m^2,q}(T_n^{n-1})\ne0\}=3m^2-2m.
\end{equation}
When $m=1$, we have $T_2^0=U_0$ satisfies \eqref{eq:rank}, and $T_2^1$ satisfies \eqref{eq:rank}\eqref{eq:low} by the description of $Kh'(T_2^1)$ in the proof of Lemma~\ref{lem:bound}.

Assume now $m>1$. We have $Kh^{m^2,*}(T_n^0)=0$ by Lemma~\ref{lem:bound}(2) applied to $S_{n-2}^{n-3}=T_n^0$; thus $T_n^0$ satisfies \eqref{eq:rank}. For $i>0$, \eqref{eq:LES_T} and \eqref{eq:T_1S_1} give
\begin{align}
\rank Kh'^{m^2,*}(T_n^i)\le&\,\rank Kh'^{m^2,*}(T_n^{i-1})+\rank Kh'^{(m-1)^2,*}(T_{n-2}^{i-1}),\quad i<n-1\label{eq:T_n^i}\\
\rank Kh'^{m^2,*}(T_n^{n-1})\le&\,\rank Kh'^{m^2,*}(T_n^{n-2})+2\rank Kh'^{(m-1)^2,*}(T_{n-2}^{n-3}).\label{eq:T_n^n-1}
\end{align}
Use \eqref{eq:T_n^i} iteratively and \eqref{eq:T_n^n-1}, as well as the induction hypothesis, we obtain $$\rank Kh'^{m^2,*}(T_n^i)\le2\binom{i}{m}.$$ On the other hand, due to the existence of the Lee spectral sequence from $Kh'\otimes\Q$ to $Kh_{Lee}$, $\rank Kh'^{m^2,*}(T_n^{n-1})$ is bounded below by $\dim Kh_{Lee}^{m^2}(T_n^{n-1})$, which equals $\binom{2m}{m}=2\binom{n-1}{m}$: $Kh_{Lee}^{m^2}(T_n^{n-1})$ is generated by those generators $[s_\mathfrak{o}]$ for which $\mathfrak{o}$ is an orientation of $T_n^{n-1}$ realizing $T(2;m,m)$ (note components in $T_n^{n-1}$ have pairwise linking number $1/2$). We conclude that $T_n^{n-1}$ satisfies \eqref{eq:rank}, and so do all $T_n^i$, because the sharpness of the estimate above shows all \eqref{eq:T_n^i}\eqref{eq:T_n^n-1} are in fact equalities.

The sharpness of estimate also implies that the map $Kh'(T_{n-2}^{n-3}\sqcup U_0)\to Kh'(T_n^{n-1})$ in the exact triangle \eqref{eq:LES_T} is injective upon tensoring $\Q$. It follows that $$\inf\{q\colon Kh'^{m^2,q}(T_n^{n-1})\ne0\}\le\inf\{q\colon Kh'^{(m-1)^2,q}(T_{n-2}^{n-3})\ne0\}-1+3n-4=3m^2-2m.$$ Lemma~\ref{lem:bound}(1) gives the reverse inequality, so \eqref{eq:low} is also proved.

We return to the calculation of the $s$-invariant. The sharpness of the estimate of $\rank Kh'(T_n^{n-1})$ also implies that the Lee spectral sequence from $Kh'(T_n^{n-1})\otimes\Q$ to $Kh_{Lee}(T_n^{n-1})$ collapses immediately at homological degree $h=m^2$. It follows that the lowest quantum filtration level of $Kh_{Lee}^{m^2}(T_n^{n-1})$ is at $q=3m^2-2m$.

Taking into account the bidegree shift $[m^2]\{3m^2\}$, the $s$-invariant of $T(2;m,m)$ is equal to the quantum filtration degree of $[s_\mathfrak{o}]\in Kh_{Lee}^{m^2}(T_n^{n-1})$ minus $3m^2-1$, where $\mathfrak{o}$ is any orientation of $T_n^{n-1}$ that realizes $T(2;m,m)$. Since $Kh_{Lee}^{m^2}(T_n^{n-1})$ is spanned by all such $[s_\mathfrak{o}]$, every $[s_\mathfrak{o}]$ sits in the lowest filtration level. It follows that $s(T(2;m,m))=(3m^2-2m)-(3m^2-1)=-2m+1$, proving Proposition~\ref{prop:s} for $|p-q|=0$.

\textbf{Case 2}: $|p-q|=2$.

Write $n=p+q=2m$. By an induction on $m$ one can show that $$\rank Kh'^{m^2-1,*}(T_n^i)=2\binom{i}{m+1}+2\binom{i}{m-1},$$ $$\inf\{q\colon Kh'^{m^2-1,q}(T_n^{n-1})\ne0\}=3m^2-2m-1.$$
Moreover, we have $\rank Kh'^{m^2-1,*}(T_n^{n-1})=\dim Kh_{Lee}(T_n^{n-1})$, which implies the collapsing of the Lee spectral sequence at $h=m^2-1$. After a bidegree shift $[m^2-1]\{3m^2-3\}$, we calculate that $s(T(2;m+1,m-1))=(3m^2-2m-1)-(3m^2-3-1)=-2m+3$, proving the case $|p-q|=2$.

\textbf{Case 3}: $|p-q|=1$.

Write $n=p+q=2m+1$. By an induction on $m$ (with base case $m=0$) one can show that $$\rank Kh'^{m^2+m,*}(T_n^i)=2\binom{i+1}{m+1},$$ $$\inf\{q\colon Kh'^{m^2+m,q}(T_n^{n-1})\ne0\}=3m^2+m-1.$$
Moreover, we also conclude the immediate collapsing of the Lee spectral sequence by a dimension count, and calculate that $s(T(2;m+1,m))=(3m^2+m-1)-(3m^2+3m-1)=-2m$.

\textbf{Case 4}: $|p-q|=3$.

Write $n=p+q=2m+1$. From $$\rank Kh'^{m^2+m-2,*}(T_n^i)=2\binom{i}{m+2}+2\binom{i}{m-1}$$ $$\inf\{q\colon Kh'^{m^2+m-2,q}(T_n^{n-1})\ne0\}=3m^2+m-3\ \ (m>0)$$ and a dimension count we conclude as above that $s(T(2;m+2,m-1))=-2m+4$. We remark that in this case one need to take both $m=0,1$ as base cases for induction, where $T_3^0=S_1^0=T_1^0=U_0$ and all $Kh'(T_3^i)$ can be completely determined from \eqref{eq:LES_T}.
\end{proof}

\subsection{A question}\label{sec:s_que}
As an analogue to Question~6.1 in \cite{ren2023lee}, we pose the following question, whose truth is verified in small examples ($n\le5$).
\begin{Que}\label{que}
Is it true that the saddle cobordism $T_n^i\to T_n^{i-1}$ always induces a surjection on $Kh'\otimes\Q$?
\end{Que}\vspace{-5pt}
A positive answer of Question~\ref{que} in the case $i=n-1$ implies the saddle cobordism $T(2;n-2,0)\sqcup U_0\to T(2;n,0)$ is injective in $Kh'\otimes\Q$. By the same argument as the proof of Theorem~1.1 ($m=n$) in \cite{ren2023lee}, this implies \eqref{eq:s_2}, thus the conjectural genus bound \eqref{eq:g_D2}. Of course, \eqref{eq:s_2} is a much weaker statement than Question~\ref{que} and would follow from the surjectivity of $$Kh'^{pq,pq+\lfloor n^2/2\rfloor-n}(T_n^{n-1})\otimes\Q\to Kh'^{pq,pq+\lfloor n^2/2\rfloor-n-1}(T_n^{n-2})\otimes\Q$$ for all $p+q=n$ (c.f. Remark~\ref{rmk:fail}).

\printbibliography

\end{document}